\documentclass[12pt]{article}
\usepackage[utf8]{inputenc}
\usepackage[T1]{fontenc}
\usepackage{url}
\usepackage{graphicx}
\usepackage{amsmath}
\usepackage{amsfonts}
\usepackage{amsthm}
\usepackage[margin=1in]{geometry}
\newcommand{\ud}{\,\mathrm{d}}
\numberwithin{equation}{section}
\begin{document}
\title{\huge{\textbf{Existence and the stability of minimizers in ferromagnetic nanowires.}}}
\author{\Large{Davit Harutyunyan}\\
\textit{Department of Mathematics of The University of Utah}}

\maketitle

$$\textbf{Abstract}$$

We study static 180 degree domain walls in infinite magnetic wires with bounded, $C^1$ and rotationally symmetric cross sections. We prove an existence of global minimizers for the energy of micromagnetics for any bounded $C^1$ cross sections. Under some asymmetry of cross sections we prove a stability result for the minimizers, namely, we show that vectors of micromagnetics having an energy close to the minimal one, must be $H^1$ close to the actual minimizer, which is itself close to the minimizer of the limit energy up to a rotation and a translation. 

 \section{Introduction}
\newtheorem{Theorem}{Theorem}[section]
\newtheorem{Lemma}[Theorem]{Lemma}
\newtheorem{Corollary}[Theorem]{Corollary}
\newtheorem{Definition}[Theorem]{Definition}

In the theory of micromagnetics to any domain $\Omega\in \mathbb R^3$ and a unit vector field (called magnetization) $m\colon\Omega\to\mathbb S^2$ with $m=0$ in $\mathbb R^3\setminus \Omega$  the energy of micromagnetics is assigned:
$$E(m)=A_{ex}\int_\Omega|\nabla m|^2+K_d\int_{\mathbb R^3}|\nabla u|^2+Q\int_\Omega\varphi(m)-2\int_\Omega H_{ext}\cdot m,$$
where $A_{ex},$ $K_d$, $Q$ are material parameters, $H_{ext}$ is the externally applied magnetic field, $\varphi$ is the anisotropy energy density and $u$ is obtained from Maxwell's equations of magnetostatics,
$$
\begin{cases}
\mathrm{curl} H_{ind}=0 & \quad\text{in}\quad \mathbb R^3\\
\mathrm{div}(H_{ind}+m)=0 & \quad\text{in}\quad \mathbb R^3,
\end{cases}
$$
i.e., $u$ is a weak solution of
$$\triangle  u= \textrm{div} m\qquad \text{in}\qquad \mathbb R^3.$$

According to micromagnetics, stable magnetization patterns are described by the minimizers of the micromagnetic energy functional, see [\ref{bib:D.S.KMO1},\ref{bib:D.S.KMO2},\ref{bib:HSch}].
The study of magnetic wires and thin films has attracted significant attention in the recent years, see [\ref{bib:AXFAPC},\ref{bib:BNKTE},\ref{bib:CL},\ref{bib:FSSSTDF},\ref{bib:HK},\ref{bib:Kuehn},\ref{bib:NT},\ref{bib:NWBKUFK},\ref{bib:PGDLFLOF},\ref{bib:Sanchez},
\ref{bib:SS},\ref{bib:WNU}] for wires and [\ref{bib:CMO},\ref{bib:D.S.KMO1},\ref{bib:D.S.KMO2},\ref{bib:GC1},\ref{bib:GC2},\ref{bib:KS1},\ref{bib:KS2},\ref{bib:Kurzke}] for thin films. It has been suggested in [\ref{bib:AXFAPC}] that magnetic nanowires can be effectively used as storage devices. When a homogenous external field is applied in the axial direction of a magnetic wire facing the homogenous magnetization direction (see Fig. 1), then at a critical strength of the field the reversal of the magnetization typically starts at one end of the wire creating a domain wall which starts moving along the wire. The domain wall separates the reversed and the not yet reversed parts of the wire (see Fig. 1). It is known that the magnetization pattern reversal time is closely related to the writing and reading speed of such a device, thus it is crucial to understand the magnetization reversal and switching processes. Several authors have numerically, experimentally and analytically observed two different magnetization modes in magnetic nanowires [\ref{bib:FSSSTDF},\ref{bib:HK},\ref{bib:Har},\ref{bib:Kuehn}]. In [\ref{bib:FSSSTDF}] the magnetization reversal process has been studied numerically in cobalt nanowires by the Landau-Lifshitz-Gilbert equation. Two different domain wall types were observed. For thin cobalt wires with 10nm in diameter the transverse mode has been observed: the magnetization is constant on each cross section and is moving along the wire. For thick wires, with diameters bigger that 20nm the vortex wall has been observed: the magnetization is approximately tangential to the boundary and forms a vortex which propagates along the wire. In [\ref{bib:HK}] the magnetization reversal process has been studied both numerically and experimentally. By considering a conical type wire so that the diameter of the cross section increases very slowly, the magnetization switching from the vortex wall to the transverse at a critical diameter has been observed, as the domain wall moves along the wire.
The results in [\ref{bib:FSSSTDF}] and [\ref{bib:HK}] were the same: in thin wires the transverse wall occurs, while in thick wires the vortex wall occurs.
\\

\setlength{\unitlength}{1.2mm}

\begin{picture}(90,28)

\put(23,27){\textbf{Homogenius magnetization}}

\put(10,10){\line(1,0){70}}
\put(80,10){\line(0,1){15}}
\put(10,25){\line(1,0){70}}
\put(10,10){\line(0,1){15}}

\thicklines
\put(11,12){\vector(1,0){5}}
\put(11,15){\vector(1,0){5}}
\put(11,18){\vector(1,0){5}}
\put(11,21){\vector(1,0){5}}
\put(11,24){\vector(1,0){5}}

\put(18,12){\vector(1,0){5}}
\put(18,15){\vector(1,0){5}}
\put(18,18){\vector(1,0){5}}
\put(18,21){\vector(1,0){5}}
\put(18,24){\vector(1,0){5}}

\put(25,12){\vector(1,0){5}}
\put(25,15){\vector(1,0){5}}
\put(25,18){\vector(1,0){5}}
\put(25,21){\vector(1,0){5}}
\put(25,24){\vector(1,0){5}}

\put(32,12){\vector(1,0){5}}
\put(32,15){\vector(1,0){5}}
\put(32,18){\vector(1,0){5}}
\put(32,21){\vector(1,0){5}}
\put(32,24){\vector(1,0){5}}

\put(39,12){\vector(1,0){5}}
\put(39,15){\vector(1,0){5}}
\put(39,18){\vector(1,0){5}}
\put(39,21){\vector(1,0){5}}
\put(39,24){\vector(1,0){5}}

\put(46,12){\vector(1,0){5}}
\put(46,15){\vector(1,0){5}}
\put(46,18){\vector(1,0){5}}
\put(46,21){\vector(1,0){5}}
\put(46,24){\vector(1,0){5}}

\put(53,12){\vector(1,0){5}}
\put(53,15){\vector(1,0){5}}
\put(53,18){\vector(1,0){5}}
\put(53,21){\vector(1,0){5}}
\put(53,24){\vector(1,0){5}}

\put(60,12){\vector(1,0){5}}
\put(60,15){\vector(1,0){5}}
\put(60,18){\vector(1,0){5}}
\put(60,21){\vector(1,0){5}}
\put(60,24){\vector(1,0){5}}

\put(67,12){\vector(1,0){5}}
\put(67,15){\vector(1,0){5}}
\put(67,18){\vector(1,0){5}}
\put(67,21){\vector(1,0){5}}
\put(67,24){\vector(1,0){5}}

\put(74,12){\vector(1,0){5}}
\put(74,15){\vector(1,0){5}}
\put(74,18){\vector(1,0){5}}
\put(74,21){\vector(1,0){5}}
\put(74,24){\vector(1,0){5}}

\end{picture}

\begin{picture}(90,26)

\put(25,27){\textbf{180 degree domain wall}}

\put(10,10){\line(1,0){70}}
\put(80,10){\line(0,1){15}}
\put(10,25){\line(1,0){70}}
\put(10,10){\line(0,1){15}}

\put(15,2){\textbf{ Figure 1.}}

\thicklines
\put(50,6){\vector(-1,0){15}}
\put(40,7){\textbf{$H_{ext}$}}

\put(11,12){\vector(1,0){5}}
\put(11,15){\vector(1,0){5}}
\put(11,18){\vector(1,0){5}}
\put(11,21){\vector(1,0){5}}
\put(11,24){\vector(1,0){5}}

\put(18,12){\vector(1,0){5}}
\put(18,15){\vector(1,0){5}}
\put(18,18){\vector(1,0){5}}
\put(18,21){\vector(1,0){5}}
\put(18,24){\vector(1,0){5}}

\put(25,12){\vector(1,0){5}}
\put(25,15){\vector(1,0){5}}
\put(25,18){\vector(1,0){5}}
\put(25,21){\vector(1,0){5}}
\put(25,24){\vector(1,0){5}}

\put(32,12){\vector(1,0){5}}
\put(32,15){\vector(1,0){5}}
\put(32,18){\vector(1,0){5}}
\put(32,21){\vector(1,0){5}}
\put(32,24){\vector(1,0){5}}

\put(39,12){\vector(1,0){5}}
\put(39,15){\vector(1,0){5}}
\put(39,18){\vector(1,0){5}}
\put(39,21){\vector(1,0){5}}
\put(39,24){\vector(1,0){5}}

\put(58,12){\vector(-1,0){5}}
\put(58,15){\vector(-1,0){5}}
\put(58,18){\vector(-1,0){5}}
\put(58,21){\vector(-1,0){5}}
\put(58,24){\vector(-1,0){5}}

\put(65,12){\vector(-1,0){5}}
\put(65,15){\vector(-1,0){5}}
\put(65,18){\vector(-1,0){5}}
\put(65,21){\vector(-1,0){5}}
\put(65,24){\vector(-1,0){5}}

\put(72,12){\vector(-1,0){5}}
\put(72,15){\vector(-1,0){5}}
\put(72,18){\vector(-1,0){5}}
\put(72,21){\vector(-1,0){5}}
\put(72,24){\vector(-1,0){5}}

\put(79,12){\vector(-1,0){5}}
\put(79,15){\vector(-1,0){5}}
\put(79,18){\vector(-1,0){5}}
\put(79,21){\vector(-1,0){5}}
\put(79,24){\vector(-1,0){5}}

\end{picture}

In Figure 2 one can see the transverse and the vortex wall longitudinal and cross section pictures for wires with a rectangular cross section. \\

\setlength{\unitlength}{1mm}
\begin{picture}(180,93)

\put(0,15){\line(0,1){71}}
\put(0,15){\line(1,0){20}}
\put(20,15){\line(0,1){71}}
\put(0,86){\line(1,0){20}}

\put(2,15){\vector(0,1){4}}
\put(5,15){\vector(0,1){4}}
\put(8,15){\vector(0,1){4}}
\put(11,15){\vector(0,1){4}}
\put(14,15){\vector(0,1){4}}
\put(17,15){\vector(0,1){4}}

\put(2,20){\vector(0,1){4}}
\put(5,20){\vector(0,1){4}}
\put(8,20){\vector(0,1){4}}
\put(11,20){\vector(0,1){4}}
\put(14,20){\vector(0,1){4}}
\put(17,20){\vector(0,1){4}}

\put(2,25){\vector(1,3){1.2}}
\put(5,25){\vector(1,3){1.2}}
\put(8,25){\vector(1,3){1.2}}
\put(11,25){\vector(1,3){1.2}}
\put(14,25){\vector(1,3){1.2}}
\put(17,25){\vector(1,3){1.2}}

\put(2,29){\vector(1,3){1.2}}
\put(5,29){\vector(1,3){1.2}}
\put(8,29){\vector(1,3){1.2}}
\put(11,29){\vector(1,3){1.2}}
\put(14,29){\vector(1,3){1.2}}
\put(17,29){\vector(1,3){1.2}}

\put(2,33){\vector(1,2){1.8}}
\put(5,33){\vector(1,2){1.8}}
\put(8,33){\vector(1,2){1.8}}
\put(11,33){\vector(1,2){1.8}}
\put(14,33){\vector(1,2){1.8}}
\put(17,33){\vector(1,2){1.8}}

\put(2,37){\vector(1,1){3.6}}
\put(5,37){\vector(1,1){3.6}}
\put(8,37){\vector(1,1){3.6}}
\put(11,37){\vector(1,1){3.6}}
\put(14,37){\vector(1,1){3.6}}
\put(17,37){\vector(1,1){3.6}}

\put(2,41){\vector(2,1){3.6}}
\put(5,41){\vector(2,1){3.6}}
\put(8,41){\vector(2,1){3.6}}
\put(11,41){\vector(2,1){3.6}}
\put(14,41){\vector(2,1){3.6}}
\put(17,41){\vector(2,1){3.6}}

\put(2,43.5){\vector(3,1){3.6}}
\put(5,43.5){\vector(3,1){3.6}}
\put(8,43.5){\vector(3,1){3.6}}
\put(11,43.5){\vector(3,1){3.6}}
\put(14,43.5){\vector(3,1){3.6}}
\put(17,43.5){\vector(3,1){3.6}}

\put(2,45.5){\vector(4,1){4}}
\put(5,45.5){\vector(4,1){4}}
\put(8,45.5){\vector(4,1){4}}
\put(11,45.5){\vector(4,1){4}}
\put(14,45.5){\vector(4,1){4}}
\put(17,45.5){\vector(4,1){4}}

\put(1,47.5){\vector(1,0){4}}
\put(6,47.5){\vector(1,0){4}}
\put(11,47.5){\vector(1,0){4}}
\put(16,47.5){\vector(1,0){4}}

\put(2,49.5){\vector(4,-1){4}}
\put(5,49.5){\vector(4,-1){4}}
\put(8,49.5){\vector(4,-1){4}}
\put(11,49.5){\vector(4,-1){4}}
\put(14,49.5){\vector(4,-1){4}}
\put(17,49.5){\vector(4,-1){4}}

\put(2,51.5){\vector(3,-1){3.6}}
\put(5,51.5){\vector(3,-1){3.6}}
\put(8,51.5){\vector(3,-1){3.6}}
\put(11,51.5){\vector(3,-1){3.6}}
\put(14,51.5){\vector(3,-1){3.6}}
\put(17,51.5){\vector(3,-1){3.6}}

\put(2,54){\vector(2,-1){3.6}}
\put(5,54){\vector(2,-1){3.6}}
\put(8,54){\vector(2,-1){3.6}}
\put(11,54){\vector(2,-1){3.6}}
\put(14,54){\vector(2,-1){3.6}}
\put(17,54){\vector(2,-1){3.6}}

\put(2,58){\vector(1,-1){3.6}}
\put(5,58){\vector(1,-1){3.6}}
\put(8,58){\vector(1,-1){3.6}}
\put(11,58){\vector(1,-1){3.6}}
\put(14,58){\vector(1,-1){3.6}}
\put(17,58){\vector(1,-1){3.6}}

\put(2,62){\vector(1,-2){1.8}}
\put(5,62){\vector(1,-2){1.8}}
\put(8,62){\vector(1,-2){1.8}}
\put(11,62){\vector(1,-2){1.8}}
\put(14,62){\vector(1,-2){1.8}}
\put(17,62){\vector(1,-2){1.8}}

\put(2,66){\vector(1,-3){1.2}}
\put(5,66){\vector(1,-3){1.2}}
\put(8,66){\vector(1,-3){1.2}}
\put(11,66){\vector(1,-3){1.2}}
\put(14,66){\vector(1,-3){1.2}}
\put(17,66){\vector(1,-3){1.2}}

\put(2,70){\vector(1,-3){1.2}}
\put(5,70){\vector(1,-3){1.2}}
\put(8,70){\vector(1,-3){1.2}}
\put(11,70){\vector(1,-3){1.2}}
\put(14,70){\vector(1,-3){1.2}}
\put(17,70){\vector(1,-3){1.2}}

\put(2,75){\vector(0,-1){4}}
\put(5,75){\vector(0,-1){4}}
\put(8,75){\vector(0,-1){4}}
\put(11,75){\vector(0,-1){4}}
\put(14,75){\vector(0,-1){4}}
\put(17,75){\vector(0,-1){4}}

\put(2,80){\vector(0,-1){4}}
\put(5,80){\vector(0,-1){4}}
\put(8,80){\vector(0,-1){4}}
\put(11,80){\vector(0,-1){4}}
\put(14,80){\vector(0,-1){4}}
\put(17,80){\vector(0,-1){4}}

\put(2,85){\vector(0,-1){4}}
\put(5,85){\vector(0,-1){4}}
\put(8,85){\vector(0,-1){4}}
\put(11,85){\vector(0,-1){4}}
\put(14,85){\vector(0,-1){4}}
\put(17,85){\vector(0,-1){4}}

%%%%%%%%%%%%%%%%%%%%%%%%%%%%%%%%%%%%%% cross section

\put(0,0){\textbf{The transverse wall}}

\put(60,0){\textbf{The vortex wall}}

\put(27,40){\line(0,1){10}}
\put(27,40){\line(1,0){21}}
\put(48,40){\line(0,1){10}}
\put(27,50){\line(1,0){21}}

\put(28,41){\vector(1,0){3}}
\put(28,43){\vector(1,0){3}}
\put(28,45){\vector(1,0){3}}
\put(28,47){\vector(1,0){3}}
\put(28,49){\vector(1,0){3}}

\put(32,41){\vector(1,0){3}}
\put(32,43){\vector(1,0){3}}
\put(32,45){\vector(1,0){3}}
\put(32,47){\vector(1,0){3}}
\put(32,49){\vector(1,0){3}}

\put(36,41){\vector(1,0){3}}
\put(36,43){\vector(1,0){3}}
\put(36,45){\vector(1,0){3}}
\put(36,47){\vector(1,0){3}}
\put(36,49){\vector(1,0){3}}

\put(40,41){\vector(1,0){3}}
\put(40,43){\vector(1,0){3}}
\put(40,45){\vector(1,0){3}}
\put(40,47){\vector(1,0){3}}
\put(40,49){\vector(1,0){3}}

\put(44,41){\vector(1,0){3}}
\put(44,43){\vector(1,0){3}}
\put(44,45){\vector(1,0){3}}
\put(44,47){\vector(1,0){3}}
\put(44,49){\vector(1,0){3}}

%%%%%%%%%%%%%%%%%%%%%%%% vortex wall

\put(55,15){\line(0,1){71}}
\put(55,15){\line(1,0){20}}
\put(55,86){\line(1,0){20}}
\put(75,15){\line(0,1){71}}

\put(55,28){\line(1,2){20}}
\put(75,28){\line(-1,2){20}}

\put(83,30){\line(0,1){30}}
\put(83,30){\line(1,0){30}}
\put(113,30){\line(0,1){30}}
\put(83,60){\line(1,0){30}}

\put(83,30){\line(1,1){30}}
\put(83,60){\line(1,-1){30}}
\put(90.5,30){\line(1,2){15}}
\put(113,37.5){\line(-2,1){30}}

\thicklines

\put(93,55){\vector(1,0){6}}
\put(105.5,55){\vector(1,-1){3.7}}

\put(103,35){\vector(-1,0){6}}
\put(90.5,35){\vector(-1,1){3.7}}

\put(88,40){\vector(0,1){6}}
\put(88,52.5){\vector(1,1){3.7}}

\put(108,50){\vector(0,-1){6}}
\put(108,37.5){\vector(-1,-1){3.7}}

\put(65,47){\vector(0,-1){5}}
\put(63.5,44){\vector(0,-1){5}}
\put(66.5,44){\vector(0,-1){5}}

\put(62,40){\vector(0,-1){5}}
\put(65,40){\vector(0,-1){5}}
\put(68,40){\vector(0,-1){5}}

\put(60.5,36){\vector(0,-1){5}}
\put(66.5,36){\vector(0,-1){5}}
\put(63.5,36){\vector(0,-1){5}}
\put(69.5,36){\vector(0,-1){5}}

\put(58,30){\vector(0,-1){5}}
\put(63,30){\vector(0,-1){5}}
\put(68,30){\vector(0,-1){5}}
\put(73,30){\vector(0,-1){5}}

\put(58,23){\vector(0,-1){5}}
\put(63,23){\vector(0,-1){5}}
\put(68,23){\vector(0,-1){5}}
\put(73,23){\vector(0,-1){5}}

%%%%%%%%%%%%%%%%%%%5 upper part

\put(65,49){\vector(0,1){5}}

\put(63.5,52){\vector(0,1){5}}
\put(66.5,52){\vector(0,1){5}}

\put(62,56){\vector(0,1){5}}
\put(65,56){\vector(0,1){5}}
\put(68,56){\vector(0,1){5}}

\put(60.5,60){\vector(0,1){5}}
\put(66.5,60){\vector(0,1){5}}
\put(63.5,60){\vector(0,1){5}}
\put(69.5,60){\vector(0,1){5}}

\put(58,66){\vector(0,1){5}}
\put(63,66){\vector(0,1){5}}
\put(68,66){\vector(0,1){5}}
\put(73,66){\vector(0,1){5}}

\put(63,73){\vector(0,1){5}}
\put(68,73){\vector(0,1){5}}
\put(73,73){\vector(0,1){5}}
\put(58,73){\vector(0,1){5}}

\put(58,80){\vector(0,1){5}}
\put(63,80){\vector(0,1){5}}
\put(68,80){\vector(0,1){5}}
\put(73,80){\vector(0,1){5}}

%%%%%%%%%%%%%%%%%%%%%%%%%5 the lateral parts

\put(62,49.5){\vector(0,1){3}}
\put(62,46.5){\vector(0,-1){3}}
\put(68,49.5){\vector(0,1){3}}
\put(68,46.5){\vector(0,-1){3}}

\put(62,48){\circle*{0.7}}
\put(68,48){\circle*{0.7}}
\put(58,48){\circle*{0.7}}
\put(72,48){\circle*{0.7}}

\put(58,46){\vector(0,-1){3}}
\put(58,41.5){\vector(0,-1){4.5}}
\put(58,50){\vector(0,1){3}}
\put(58,55){\vector(0,1){4.5}}

\put(72,46){\vector(0,-1){3}}
\put(72,41.5){\vector(0,-1){4.5}}
\put(72,50){\vector(0,1){3}}
\put(72,55){\vector(0,1){4.5}}

%%%%%%%%%%%%%%%%%%%%%%%%%%%% the cross section

\put(5,7){\textbf{ Figure 2.}}\\
\end{picture}
\\
\\

In [\ref{bib:Har1}] the author studies a similar problem for thin films and derives $\Gamma$-convergence result for the energies. In a series of papers the authors study the magnetization reversal process in thin films, identifying four different scaling regimes for the critical value of the applied external field, see [\ref{bib:Con.Ott.1},\ref{bib:Con.Ott.2},\ref{bib:Con.Ott.Ste.1},\ref{bib:Ot.St.},\ref{bib:St.Sc.Wi.Mc.Ot.}]. In nanowires, it has been observed that there is a distinctive crossover between two different modes, which occurs at a critical diameter of the wire and it was suggested that the magnetization switching process can be understood by analyzing the micromagnetics energy minimization problem for different diameters of the cross section. In [\ref{bib:Kuehn}], K. K\"uhn studied $180$ degree static domain walls in magnetic wires with circular cross sections by an asymptotic analysis proving that indeed the transverse mode must occur in thin magnetic wires. It is also shown in [16] that for thick wires the vortex wall has the optimal energy scaling and that the minimal energy scales like $R^2\sqrt{\ln R}.$ In [\ref{bib:SS}] V.V.Slastikov and C.Sonnenberg studied a similar problem for finite curved wires proving a $\Gamma$-convergence on energies as the diameter of the wire goes to zero. In [\ref{bib:Har}], the author studied the same problem as K.K\"uhn in [\ref{bib:Kuehn}] and independently of [\ref{bib:SS}] (see the submission and the publication dates of [\ref{bib:Har}] and [\ref{bib:SS}] respectively) extended some of the results proven in [\ref{bib:Kuehn}] for arbitrary wires with a rotational symmetry. In this paper we study the $180$ degree static domain walls in magnetic wires with arbitrary bounded, $C^1$ and rotationally symmetric cross sections. We generalize the existence of minimizers result proven by K.K\"uhn for circular cross sections, to wires with arbitrary bounded and $C^1$ cross sections. For a class of domains we prove a stability of minimizers result, which is a new and much deeper result and it does not follow from the $\Gamma-$convergence of the energies. It actually requires much deeper analysis of minimization problem of minimizing the energy of micromagnetics and its minimizers.

\section{The main results}

Assume $\Omega=\mathbb R\times \omega$, where $\omega\subset \mathbb R^2$ is a bounded $C^1$ domain. Consider the isotropic energy of micromagnetics without an external field like in [\ref{bib:Kuehn},\ref{bib:SS},\ref{bib:Har}],
$$E(m)=A_{ex}\int_\Omega|\nabla m|^2+K_d\int_{\mathbb R}|\nabla u|^2.$$
By scaling of all coordinates one can achieve the situation where $A_{ex}=K_d,$ so we will henceforth assume that $A_{ex}=K_d=1.$
Next we rescale the magnetization $m$ in the $y$ and $z$ coordinates such that the domain of the rescaled magnetization is fixed, i.e., if $d=\mathrm{diam}(\omega),$ then set $\acute m(x,y,z)=m(x,dy,dz).$

Denote
$$A(\Omega)=\{m\colon\Omega\to\mathbb S^2 \ : \ m\in H_{loc}^1(\Omega), \ E(m)<\infty\}.$$
We are interested in $180$ degree domain walls, so set
$$\tilde A(\Omega)=\{m\colon\Omega\to\mathbb S^2 \ : \ m-\bar e\in H^1(\Omega)\},$$
where
\begin{equation*}
\bar e(x,y,z) = \left\{
\begin{array}{rl}
(-1,0,0) & \text{if } \ \ x<-1 \\
(x,0,0)  & \text{if } \ \ -1\leq x \leq 1 \\
(1,0,0) & \text{if } \ \ 1<x \\
\end{array} \right.
\end{equation*}
The objective of this work will be studying the existence of minimizers in the minimization problem
\begin{equation}
\label{minimization problem}
\inf_{m\in \tilde A(\Omega)}E(m),
\end{equation}
and the behavior of its almost minimizers, where the notion of "almost minimizers" will be defined later in Definition~\ref{def:almost.min}.
The following existence theorem is a generalization of the corresponding theorem proven for circular cross sections in [\ref{bib:Kuehn}].

 \begin{Theorem}[Existence of minimizers]
\label{th:existence}
For every bounded $C^1$ domain $\omega\in\mathbb R^2$ there exists a minimizer of $E$ is $\tilde A(\Omega).$
\end{Theorem}

It has been shown for circular wires in [\ref{bib:Kuehn}] and later for any cross sections in [\ref{bib:SS}] and for cross sections with a rotational symmetry in [\ref{bib:Har}], that as $d$ goes to zero, the rescaled energy functional $\frac{E(m)}{d^2}$, $\Gamma$-converges to a one dimensional energy $E_0(m^0)$ under the following notion of convergence of magnetization vectors:

\begin{Definition}
\label{notion of convergence}
The sequence  $\{\acute m^n\}\subset A(\Omega)$ is said to converge to $m^0$ as $n$ goes to infinity if,
\begin{itemize}
\item[(i)] $\nabla \acute m^n\rightharpoonup\nabla m^0 $ \ \  weakly in \ \ $L^2(\Omega)$
\item[(ii)] $\acute m^n \rightarrow m^0$ \ \ strongly in \ \ $L_{loc}^2(\Omega).$
\end{itemize}

\end{Definition}
The limit or reduced energy is given by
\begin{equation}
\label{linit.energy}
E_0(m)=
\begin{cases}
|\omega|\int_{\mathbb{R}}|\partial_x m|^2\ud x+\int_{\mathbb{R}}m M_\omega m^T\ud x,&\quad \text{if}\quad m=m(x),\\
\infty,&\qquad \text{otherwise},
\end{cases}
\end{equation}
Where $M_\omega$ is a symmetric matrix given by
$$M_\omega=-\frac{1}{2\pi}\int_{\partial \omega}\int_{\partial \omega} n(x)\otimes n(y)\ln |x-y|\ud x\ud y,$$
and $n=(0,n_2,n_3)$ is the outward unit normal to $\partial\omega,$ see [\ref{bib:SS}].
Since $M_\omega$ is symmetric it can be diagonalized by a rotation in the $OYZ$ plane. We choose the coordinate system such that
$M_\omega$ is diagonal. Assume now $\omega$ is fixed and $\mathrm{diam}(\omega)=1.$ Actually, the $\Gamma$-convergence theorem implies the following two properties of the minimal
energies and sequences of minimizers:
\begin{itemize}
\item[(i)]
\begin{equation}
\label{conv.energies}
\lim_{d\to 0}\min_{m\in \tilde A(d\cdot\Omega)}\frac{E(m)}{d^2}=\min_{m\in A_0}E_0(m),
\end{equation}
where $A_0=\{m\colon\mathbb R\to \mathbb R^3 : |m|=1,\  m(\pm\infty)=\pm 1 \}.$

\item[(ii)] If $m^n$ is any sequence of minimizers with $m^n$ defined in $d\cdot\Omega,$ then a subsequence of $\{\acute m^n\}$
converges to a minimizer of $E_0$ in the sense of Definition~\ref{notion of convergence}.
\end{itemize}

It turns out, that under some asymmetry condition
on $\omega$ a stronger convergence holds, namely an $H^1$ convergence of the whole sequence of almost minimizers holds.

\begin{Definition}
\label{def:almost.min}
Let $\{d_n\}$ be a sequence of positive numbers such that $d_n\to 0.$ A sequence of magnetizations $\{m^n\}$ defined in $d_n\cdot \Omega$
is called a sequence of almost minimizers if
\begin{equation}
\label{almost.min}
\lim_{n\to \infty}\frac{E(m^n)}{d_n^2}=\min_{m \in A_0}E(m).
\end{equation}
\end{Definition}

We are now ready to formulate the other result of the paper.
\begin{Theorem}[Convergence of almost minimizers]
\label{th:almost.minimizers}
Let $\{d_n\}$ be a sequence of positive numbers such that $d_n\to 0.$ Assume that the domain $\omega$ is so that $M_\omega$ has three different eigenvalues. Then for any sequence of almost minimizers $\{m^n\}$ defined in $d_n\cdot\Omega,$
there exist a sequence $\{T_n\}$ of translations in the $x$ direction and a sequence $\{R_n\}$ of rotations in the $OYZ$ plane, each of which is either the identity or the rotation by $180$ degrees such that for $\tilde m^n(x,y,z)=m^n(T_n(R_n(x,y,z)))$ for a minimizer $m^\omega$ of $E_0$, there holds,
$$\lim_{n\to\infty}\frac{1}{d_n}\|\tilde m^n-m^\omega\|_{H^1(\Omega_n)}=0.$$
\end{Theorem}
We refer to Appendix for the definition of $m^\omega.$
The convergence in the above theorem actually states the stability of minimizers in nanowires, i.e., when the energy of a magnetization is close to the minimal one, then the magnetization vector must be close to the actual minimizer.

\section{The oscillation preventing lemma}
In this section we prove a lemma, that will be crucial in proving both existence and convergence of almost minimizers results. The lemma
 bounds the oscillations of a magnetization $m$ and the total measure of the set where $m$ develops oscillations by the
energy of $m.$ Uzing the idea of Kohn and Slastikov in [\ref{bib:KS2}] of the dimension reduction in thin domains, define

$$\bar m(x)=\int_{\omega}m(x,y,z)\ud y\ud z.$$
 Using the definition of $M_\omega^1$ it is straightforward to show that $M_\omega^1$ is positive definite,
 where $M_\omega^1$ is the lower right $2\times2$ block of $M_\omega.$
  Denote for convenience
 $$M_\omega^1=
 \begin{bmatrix}
 \alpha_2 & 0\\
 0 &  \alpha_3
 \end{bmatrix}.
$$
It has been explicitly shown in [\ref{bib:Har}, Corollary 3.7.5] and implicitly in [\ref{bib:SS}, Proof of Lemma 4.1], that the inequality below holds uniformly in $m\colon (d\cdot\Omega)\to\mathbb S^2$:
\begin{equation}
\label{lower.bound.E}
\frac{E(m)}{d^2}\geq \int_{\Omega}|\nabla m|^2+\alpha_2\int_{\mathbb R}|\bar m_2|^2+\alpha_3\int_{\mathbb R}|\bar m_3|^2+o(1),
\end{equation}
as $d$ goes to zero.
\begin{Lemma}
\label{lem:m2.m.3.bdd.E}
Assume $m^d\in A(d\cdot\Omega).$ Then there exists $d_0>0$ such that,
\begin{align}
\label{bar.m.d<d0}
&\int_{\mathbb R}(|\bar m_2|^2+|\bar m_3|^2)\leq \frac{2E(m)}{d^2\min(\alpha_2,\alpha_3)},\quad\text{if}\quad d\leq d_0\\
\label{bar.m.d>d0}
&\int_{\mathbb R}(|\bar m_2|^2+|\bar m_3|^2)\leq \frac{2\max\left(\frac{d}{d_0},(\frac{d}{d_0})^3\right)E(m^d)}{dd_0\min(\alpha_2,\alpha_3)},\quad\text{if}\quad d>d_0
\end{align}
\end{Lemma}

\begin{proof}
Due to inequality (\ref{lower.bound.E}) there exists $d_0>0$ such that for $d\leq d_0$ we have
$$\frac{2E(m)}{d^2}\geq \alpha_2\int_{\mathbb R}|\bar m_2|^2+\alpha_3\int_{\mathbb R}|\bar m_3|^2,$$
and inequality (\ref{bar.m.d<d0}) follows. Assume now $d>d_0.$ It is straightforward that if $m^d\in A(d\cdot\Omega)$ then
$m_t^d(x,y,z)=m^d(tx,ty,tz)\in A(\frac{d}{t}\cdot\Omega)$ with $E(m_t^d)=tE_{ex}(m^d)+t^3E_{mag}(m^d)$, where
$E_{ex}(m)=\int_{\Omega}|\nabla m|^2$ is the exchange energy and $E_{mag}(m)=\int_{\mathbb R^3}|\nabla u|^2$ is the magnetostatic energy,
thus we get on one hand,
\begin{equation}
\label{E(m).E(mt)}
E(m_t^d)\leq \max(t,t^3)E(m^d).
\end{equation}
But on the other hand we have

$$\int_{\mathbb R}(|\bar m_2^d|^2+|\bar m_3^d|^2)=\frac{1}{t}\int_{\mathbb R}(|\bar m_{t2}^d|^2+|\bar m_{t3}^d|^2),$$
thus we obtain choosing $t=\frac{d}{d_0}$ and taking into account (\ref{bar.m.d<d0}) and (\ref{E(m).E(mt)}),
$$\int_{\mathbb R}(|\bar m_2^d|^2+|\bar m_3^d|^2)\leq\frac{2d_0E(m_t^d)}{dd_0^2\min(\alpha_2,\alpha_3)}\leq \frac{2\max\left(\frac{d}{d_0},(\frac{d}{d_0})^3\right)E(m^d)}{dd_0\min(\alpha_2,\alpha_3)}$$
which completes the proof.

\end{proof}
Next we prove a simple estimate between $m$ and $\bar m$ that will be useful in the proof of the oscillation preventing lemma.
\begin{Lemma}
\label{lem:ineq.m.bar.m}

For any $m\in A(\Omega)$ there holds
$$\int_{\omega}(|m|^2-|\bar m|^2)=\int_{\omega}|m-\bar m|^2\leq C_pd^2\int_{\omega}|\nabla_{yz}m|\qquad  \text{for all}\qquad x\in \mathbb{R},
$$
where $C_p$ is the Poincar\'e constant of $\omega.$
\end{Lemma}

\begin{proof}
We have for any $x\in\mathbb{R}$
$$\int_{\omega}(m-\bar  m)=\int_{\omega}m-|\omega|\cdot \bar m(x)=0,$$
thus by the Poincer\'e inequality we get
\begin{align*}
\int_{\omega}|m|^2&=\int_{\omega}|\bar m|^2+\int_{\omega}|m-\bar m|^2+2\bar m(x)\int_{\omega}(m-\bar m)\\
&=\int_{\omega}|\bar m|^2+\int_{\omega}|m-\bar m|^2\\
&\leq \int_{\omega}|\bar m|^2+C_pd^2\int_{\omega}|\nabla_{yz}m|,
\end{align*}
the proof is complete now.

\end{proof}

\begin{Lemma}[Oscillation preventing lemma]
\label{lem:oscilation.preventing}

Let $m\in A(\Omega)$ and let $\alpha,\beta,\rho\in \mathbb R$ such that $-1<\alpha<\beta<1$ and $0<\rho<1.$ Assume $\Re$ is a family of disjoint intervals $(a,b)$ satisfying the conditions
$$\{\bar m_1(a), \bar m_1(b)\}=\{\alpha,\beta\}\qquad \text{and}\qquad |\bar m_1(x)|\leq \rho,\qquad  x\in (a,b).$$
Then,
\begin{itemize}

\item[(i)]
\begin{equation}
\label{card.sum}
\mathrm{card}(\Re)\leq M \qquad  \text{and}\qquad \sum_{(a,b)\in\Re}(b-a)\leq M,
\end{equation}
where $M$ is a constant depending on $\alpha$, $\beta,$ $\rho,$ $\omega$ and $E(m)$.
\item[(ii)] The component $\bar m_1,$ satisfies
$\lim_{x\to\pm\infty}|\bar m_1(x)|=1.$
\end{itemize}

\end{Lemma}

\begin{proof}

Let us first prove the second inequality in (\ref{card.sum}). The function  $\bar m$ is a one variable weakly differentiable function therefore it is locally absolutely continuous in $\mathbb{R}.$ For any $(a,b)\in \Re,$ we have by Lemma~\ref{lem:ineq.m.bar.m} and by the assumption of the lemma,

\begin{align*}
|\omega|(b-a)&=\int_{(a,b)\times \omega}|m|^2\\
&\leq \int_{(a,b)\times \omega}|\bar m|^2+C_pd^2\int_{(a,b)\times \omega}|\nabla_{yz}m|^2\\
&\leq \rho^2|\omega|(b-a)+\int_{(a,b)\times \omega}(\bar m_2^2+\bar m_3^2)+C_pd^2\int_{(a,b)\times \omega}|\nabla m|^2.
\end{align*}

Summing up the inequalities for all $(a,b)\in \Re$ we get,

\begin{align*}
|\omega|\cdot\sum_{(a,b)\in\Re}(b-a)&\leq \rho^2|\omega|\sum_{(a,b)\in\Re}(b-a)+\int_{\Sigma}(\bar m_2^2+\bar m_3^2)+
C_pd^2\int_{\Sigma}|\nabla m|^2\\
&\leq \rho^2|\omega|\sum_{(a,b)\in\Re}(b-a)+\int_{\Omega}(\bar m_2^2+\bar m_3^2)+C_pd^2\int_{\Omega}|\nabla m|^2,
\end{align*}
where $\Sigma=\bigcup_{(a,b)\in\Re}(a,b)\times \omega.$
By virtue of Lemma~\ref{lem:m2.m.3.bdd.E} we have
$$\int_{\Omega}(\bar m_2^2+\bar m_3^2)\leq C_1,$$
for some $C_1$ depending on $\omega$ and $E(m).$  Therefore we obtain
\begin{equation}
\label{sum (b-a)}
\sum_{(a,b)\in\Re}(b-a)\leq \frac{C_1+C_pd^2E(m)}{|\omega|(1-\rho^2)}.
\end{equation}
Next we have for any point $(y,z)\in \omega $ and any interval $(a,b)\in \Re,$
$$
\int_a^b|\partial_xm_1(x,y,z)|^2\ud x\geq \frac{1}{b-a}\bigg(\int_a^b|\partial_xm_1(x,y,z)|\ud x\bigg)^2,
$$
Thus integrating over $\omega$ we get
\begin{align*}
\int_{(a,b)\times \omega}|\partial_xm_1|^2\ud \xi&\geq \frac{1}{b-a}\int_{\omega}\bigg(\int_a^b|\partial_xm_1(x,y,z)|\ud x\bigg)^2\ud y\ud z\\
&\geq\frac{1}{b-a}\int_{\omega}|m_1(a,y,z)-m_1(b,y,z)|^2\ud y\ud z\\
&\geq\frac{1}{|\omega|(b-a)}\bigg(\int_{\omega}\big(m_1(a,y,z)-m_1(b,y,z)\big)\ud y\ud z\bigg)^2\\
&=\frac{|\omega|(\alpha-\beta)^2}{b-a},
\end{align*}
thus
$$\int_{(a,b)\times \omega}|\partial_xm_1|^2\ud \xi\geq\frac{|\omega|(\alpha-\beta)^2}{b-a}.$$
Summing up the last inequalities for all $(a,b)\in \Re$ we arrive at

\begin{align*}
\sum_{(a,b)\in \Re}\frac{1}{b-a}&\leq \frac{1}{|\omega|(\alpha-\beta)^2}\int_{\Sigma}|\partial_xm_1|^2\ud \xi\\
&\leq\frac{1}{|\omega|(\alpha-\beta)^2}\int_{\Omega}|\nabla m|^2\ud \xi\\
&\leq \frac{E(m)}{|\omega|(\alpha-\beta)^2},
\end{align*}
thus
\begin{equation}
\label{sum 1/(b-a)}
\sum_{(a,b)\in \Re}\frac{1}{b-a}\leq \frac{E(m)}{|\omega|(\alpha-\beta)^2}.
\end{equation}
Combining now (\ref{sum (b-a)}) and (\ref{sum 1/(b-a)}) we obtain,
\begin{equation}
\label{(a,b).finite.estimate}
\sum_{(a,b)\in \Re}\bigg(\frac{1}{b-a}+b-a\bigg)\leq\frac{1}{|\omega|}\bigg(\frac{E(m)}{(\alpha-\beta)^2}+
\frac{C_1+C_pd^2E(m)}{1-\rho^2}\bigg):=M(\alpha,\beta,\rho,\omega,E(m)).
\end{equation}
The last inequality and the inequality  $\frac{1}{b-a}+b-a\geq 2$ yield $M(\alpha,\beta,\rho,\omega,E(m))\geq 2\mathrm{card}(\Re),$ which finishes the proof of the first part. It is clear that
 $$|\bar m_1(x)|=\frac{1}{|\omega|}\bigg|\int_{\omega}m_1(x,y,z)\ud y\ud z\bigg| \leq\frac{1}{|\omega|}\int_{\omega}|m_1(x,y,z)|\ud y\ud z\leq 1$$
thus
$$0\leq 1-\bar m_1^2(x)\leq 1,\qquad  x\in\mathbb{R}.$$
By virtue of Lemma~\ref{lem:m2.m.3.bdd.E} and Lemma~\ref{lem:ineq.m.bar.m} we have,

$$\int_{\Omega}(1-\bar m_1^2)\ud \xi\leq\int_{\Omega}(\bar m_2^2+\bar m_3^2)\ud \xi+C_pd^2E(m)<\infty,$$
thus
\begin{equation}
\label{1-bar m_x^2 has finite norm}
\int_{\mathbb{R}}(1-\bar m_x^2)\ud x<\infty.
\end{equation}
The integrand is continuous and positive thus for any  $0<\delta<1$ and $N>0$ there exists $x_\delta>N$ such that $|\bar m_1(x_\delta)|>1-\frac{\delta}{2}$. Therefore there exists an increasing sequence $\{x_n\}$ such that $x_n\to\infty$ and $|\bar m_1(x_n)|>1-\frac{\delta}{2}$. Thus for infinitely many indices $n$ one has one of the following: $\bar m_1(x_n)>1-\frac{\delta}{2}$\  or \  $\bar m_1(x_n)<-1+\frac{\delta}{2}$. Assume that for a subsequence (not relabeled) there holds $\bar m_1(x_n)>1-\frac{\delta}{2}$. Let us then show, that
$\bar m_1(x)>1-\delta$ for all $x>N_{\delta}$ and some $N_{\delta}$. Assume in the contrary that for an increasing sequence $(\tilde x_n)_{n\in\mathbb{N}}$ \ with $\tilde x_n\to\infty$ one has $\bar m_1(\tilde x_n)\leq 1-\delta$. We construct an infinite family of disjoint intervals $(a_n,b_n)$ such that the value of $\bar m_1$ at one end of $(a_n,b_n)$ is less or equal than $1-\delta$ and at the other end is bigger than $1-\frac{\delta}{2}$ for all $n\in \mathbb{N}$. We start with taking the smallest $n$ such that $\tilde x_n>x_1$ and denote it by $\tilde n_1$ and set $a_1=x_1$, $b_1=\tilde x_{\tilde n_1}$. In the second step we take the smallest $n$ such that $x_n>b_1$ and denote it by $n_2$ and then we take the smallest $n$ such that
$\tilde x_n>x_{n_2}$ and denote it by $\tilde n_2$ and set $a_2=x_{n_2}$ and $b_2=\tilde x_{\tilde n_2}$. This process will never stop, thus the intervals $(a_n,b_n)$ are constructed such that $\bar m_x(a_n)>1-\frac{\delta}{2}$ \ and
\ $\bar m_x(b_n)<1-\delta.$ Since $\bar m_x$ is continuous in $\mathbb{R}$ the new sequence of disjoint intervals $(\acute a_n, \acute b_n)$ where $\acute a_n=\sup\{ x\in (a_n,b_n) \ | \ \bar m_x(x)\geq 1-\frac{\delta}{2}\}$ and $\acute b_n=\inf\{ x\in (\acute a_n,b_n) \ | \ \bar m_x(x)\leq 1-\delta\}$ have the properties $\bar m_1(\acute a_n)=1-\frac{\delta}{2},$\  $\bar m_1(\acute b_n)=1-\delta$ and $|\bar m_x(x)|\leq 1-\frac{\delta}{2} $ for all $x\in[\acute a_n,\acute b_n]$ which contradicts (\ref{card.sum}). The same can be done for $-\infty$.

\end{proof}

\newtheorem{Remark}[Theorem]{Remark}
\begin{Remark}
\label{rem:lim.bar.m1=1}
If $m\in\tilde A(\Omega)$ then $\lim_{x\to\pm\infty}\bar m_1(x)=\pm 1.$
\end{Remark}

\begin{proof}
By Lemma~\ref{lem:oscilation.preventing} we have $\lim_{x\to\pm\infty}|\bar m_1(x)|=1.$ Since $\bar m_1(x)$ is continuous and $\bar m-\bar e\in H^1(\Omega),$
then the proof follows.
\end{proof}

\begin{Remark}
\label{rem:lim.m0}
If $|m|=1$ and $E_0(m)<\infty$ then $\lim_{x\to\pm\infty}|m_1(x)|=1.$
\end{Remark}

\begin{proof}
The proof is analogues to the proof of property $(ii)$ in Lemma~\ref{lem:oscilation.preventing}.
\end{proof}

\section{Existence of minimizers}

We start by proving a simple compactness lemma that will be crucial in the proof of the existence theorem.

\begin{Lemma}
\label{lem:compactness}
Assume that the sequence of magnetizations  $\{m^n\}$ defined in the same domain $\Omega$ satisfies and $E(m^n)\leq C$ for some constant $C.$ Then there exists a magnetization $m^0\colon \Omega \rightarrow \mathbb{S}^2 $ such that for a subsequence of $\{m^n\}$ (not relabeled) the following statements hold:
\begin{itemize}

\item[(i)] $\nabla m^n\rightharpoonup\nabla m^0$ weakly in $L^2(\Omega)$
\item[(ii)] $m^n\rightarrow m^0$ strongly in $L_{loc}^2(\Omega)$
\item[(iii)] $E(m^0)\leq \liminf E(m^n)$.

\end{itemize}

\end{Lemma}

\begin{proof}

Let $u_n$ be a weak solution of $\triangle u=\mathrm{div}m^n$. From $\int_{\Omega}|\nabla m^n|^2+\int_{\mathbb R^3}|\nabla u^n|^2\leq C$ we get by a standard compactness argument that,
 $\nabla m^n\rightharpoonup \nabla m^0$ in $ L^2(\Omega),$ $\nabla u_n\rightharpoonup g$ in $L^2(\mathbb{R}^3)$ and $m^n\rightarrow m^0$ in $L_{loc}^2(\Omega),$ for the same subsequence (not relabeled) of $\{\nabla m^n\}$ and $\{\nabla u_n\}$ and some $f\in L^2(\Omega)$ and $ g\in L^2(\mathbb{R}^3).$ We extend $m^0$ outside $\Omega$ as zero. The identities
 $$\int_{\Omega} m^n\cdot \nabla\varphi=\int_{\mathbb{R}^3} \nabla u_n\cdot \nabla\varphi\quad \text{for all}\quad n\in\mathbb N\quad\text{and}\quad\varphi\in C_0^\infty(\mathbb R^3),$$
 will then yield
$$\int_{\Omega} m^0\cdot \nabla\varphi =\int_{\mathbb{R}^3} g\cdot \nabla\varphi\quad \text{for all}\quad  \varphi\in C_0^\infty(\mathbb R^3).$$
Since $g\in L^2(\mathbb{R}^3)$ then the Helmoltz projection of $g$ onto the subspace of gradient fields in $L^2(\mathbb R^3)$ will have
the form $\nabla u_0,$ will satisfy $\|\nabla u_0\|_{L^2(\mathbb{R}^3)}\leq\|g\|_{L^2(\mathbb{R}^3)}$ and will be a weak solution of $\triangle u=\mathrm{div} g$ which is equivalent to
$$\int_{\mathbb{R}^3} g\cdot \nabla\varphi =\int_{\mathbb{R}^3} \nabla u_0\cdot \nabla\varphi  \quad
\text{for all} \quad\varphi\in C_0^{\infty}(\mathbb{R}^3),$$
thus we get
$$\int_{\mathbb{R}^3} m^0\cdot \nabla\varphi=\int_{\mathbb{R}^3} \nabla u_0\cdot \nabla\varphi\quad
\text{for all} \quad \varphi\in C_0^{\infty}(\mathbb{R}^3)$$
 which means that $u_0$ is a weak solution of
 $$\triangle u=\mathrm{div} m^0.$$

  Therefore from the weak convergence $\nabla m^n\rightharpoonup \nabla m^0$ and $\nabla u_n\rightharpoonup g$ we obtain,
\begin{align*}
\|\nabla u_0\|_{L^2(\mathbb{R}^3)}&\leq\|g\|_{L^2(\mathbb{R}^3)}\leq \liminf_{n\to\infty} \|\nabla u_n\|_{L^2(\mathbb{R}^3)}\\
\|\nabla m^0\|_{L^2(\mathbb{R}^3)}&\leq \liminf_{n\to\infty} \|\nabla m^n\|_{L^2(\mathbb{R}^3)}
\end{align*}
which yields $E(m^0)\leq \liminf_{n\to\infty} E(m^n).$

\end{proof}

Now we have enough tolls to prove the existence theorem.\\

\textbf{Proof of Theorem~\ref{th:existence}}. We adopt the direct method of proving an existence of a minimizer. The idea is starting with
any minimizing sequence, we construct another minimizing sequence that has a limit in $\tilde A(\Omega)$ in the sense of Lemma~\ref{lem:compactness}.
Let $\{m^n\}$ be a minimizing sequence, i.e.,
$$\lim_{n\rightarrow \infty}E(m^n)=\inf_{m\in\tilde A(\Omega)}E(m).$$
First of all note that minimization problem (\ref{minimization problem}) is invariant under translations in the $x$ direction, that is if $m\in \tilde A(\Omega)$ then obviously $m_c(x,y,z)=m(x-c,y,z)\in \tilde A(\Omega)$ and $E(m_c)=E(m).$ We have that $|E(m^n)|\leq M$ for some $M$ and for all $n \in \mathbb{N}$. For any $n\in\mathbb{N}$ consider the sets $A_n$, $B_n$ and $C_n$ defined as follows:

\begin{align*}
A_n&=\bigg\{x\in \mathbb{R} \ \ : \ -1\leq \bar m_1^n(x)< -\frac{1}{2}\bigg\}\\
B_n&=\bigg\{x\in \mathbb{R} \ \ : \ -\frac{1}{2}\leq \bar m_1^n(x)\leq \frac{1}{2}\bigg\}\\
C_n&=\bigg\{x\in \mathbb{R} \ \ : \  \frac{1}{2}<\bar m_1^n(x)\leq 1\bigg\}\\
\end{align*}

Since $\bar m_1^n$ is continuous in $\mathbb{R}$ then for all $n \in\mathbb{N}$, $A_n,$ $B_n$ and $C_n$ are a finite or countable union of disjoint intervals. We distinguish two types of intervals in $B_n.$ A composite interval $(a,b)\subset B_n$ is said to be of the first type if $|\bar m_1^n(a)-\bar m_1^n(b)|=1,$  and of the second type otherwise.
By Lemma~\ref{lem:oscilation.preventing} the sum of the lengths of all intervals, as well as the number of the first type intervals in $B_n$ is bounded by a number $s$ depending only on $M$ and $\omega$, i.e., a constant not depending on $n$. Consider two cases:\\
\textbf{CASE1.} \textit{There are no second type intervals in $B_n$ for all $n\in\mathbb{N}.$}\\

Let us paint all the points of $A_n$, $B_n$ and $C_n$ with respectively black, yellow and red color for all $n\in \mathbb{N}$. We call the increasing sequence $\{n_k\}\subset \mathbb N$ "good" if for every $k\in\mathbb{N}$ there exist two intervals $(a_1^k,a_2^k)\subset A_{n_k}$ and
$(c_1^k,c_2^k)\subset C_{n_k}$ such that
$$a_2^k-a_1^k\rightarrow +\infty,\qquad c_2^k-c_1^k\rightarrow +\infty,\qquad 0<c_1^k-a_2^k\leq C$$
for a constant $C$ not depending on $k.$ The endpoints $a_1^k$ and $c_2^k$ can also take values
$-\infty$ and $+\infty$ respectively. If $\{n_k\}$ is "good", the subsequence $\{m^{n_k}\}$ will also be called "good". We show, that any minimizing sequence $\{m^n\}\subset\tilde A(\Omega)$ can be translated in the $x$ coordinate such that the new sequence contains a "good" subsequence. For every fixed $n$ there are some black, yellow and red intervals between $(-\infty, a_n)$ and $(c_n,+\infty)$. Note that there is obviously at least one yellow interval between any two black and any two red ones, thus the number of both black and yellow intervals is at most $s+1$, hence the number of all intervals in the $n$-th family is bounded by the same number $S=3s+2$ for all $n.$ Let us number both the red and the black intervals in any family of intervals. Let us prove the proposition below, which is a reformulation of our problem:\\

\textbf{Proposition.} Assume a sequence of natural numbers $l_n$ and a sequence of families of $l_n$ disjoint intervals on the real line pained with black and red color are given for all $n\in\mathbb{N}$. Assume $l_n\leq l$ and the sum of the lengths of $l_n-1$ gaps between the intervals in the $n$-th family is bounded by the same number $M$ for all $n$. Assume furthermore that for any $n,$ the far left placed interval is black and the far right placed interval is red and their lengths tend to $\infty$ as $n$ goes to infinity. Then there exists a subsequence $\{n_k\}$ and two associated intervals $(a_1^k, a_2^k)$ and
 $(c_1^k, c_2^k)$ in the $n_k$-th family such that $(a_1^k, a_2^k)$ is black, $(c_1^k, c_2^k)$ is red, and

\begin{equation}
a_2^k-a_1^k \rightarrow +\infty,\qquad c_2^k-c_1^k \rightarrow +\infty \qquad 0<c_1^k-a_2^k\leq M_1
\end{equation}
for a constant $M_1$ and all $k\in \mathbb{N}$.\\

\textbf{Proof of proposition.} The case  $l=2$ is evident. Assume that the proposition is true for $l\leq N$ and let us prove it for $l=N+1$. Since $l\geq 3$, in every family there are at least two intervals of the same color. Assume that for infinitely many indices $n$ there are at least two black intervals in the $n$-th family. Consider the far right placed black intervals for all such families. There are two possible cases:\\
 \textbf{Case 1.} \textit{For a subsequence their lengths tend to $+\infty$}.\\
 In this case we can omit all the intervals placed on their left side which leads to a situation with less intervals in every family (in such a subsequence) fulfilling the requirements of the proposition, so by induction the existence of a "good" subsequence is proven.\\
 \textbf{Case 2.} \textit{Their lengths are bounded by the same constant.}\\
 In this case we can remove this intervals and this will lead us to a situation with less intervals in all families fulfilling the requirements of the statements so by the induction the existence of a "good" subsequence is proven.\\

 Let us get now back to our situation. If we remove all the yellow intervals from the real line for all $n\in \mathbb{N}$ then the families of the black and the red intervals fulfill the requirements of the proposition, thus the existence of a "good" sequence is proven. Take the two intervals $[a_1^k,a_2^k]$ and $[c_1^k,c_2^k]$ for all $k\in\mathbb{N}$ and denote the the "good" subsequence of magnetizations again by $\{m^k\}$ which will also be a minimizing sequence. Let us translate $m^k$ by a factor of $a_2^k$ and denote
 $$m_{good}^k(x,y,z)=m^k(x+a_2^k,y,z).$$
Then $\{m_{good}^k\}$ is a minimizing sequence and furthermore denoting $a_3^k=a_2^k-a_1^k$, $c_3^k=c_1^k-a_2^k$ and $c_4^k=c_2^k-a_2^k,$ we obtain,
\begin{align}
\label{conditions.good1}
 &\bar m_{good}^k(x)\leq -\frac{1}{2} \quad\text{for} \quad x \in [-a_3^k,0] \quad \text{and} \quad\bar m_{good}^k(x)\geq \frac{1}{2} \quad\text{for}\quad x \in [c_3^k,c_4^k],\\
&a_3^k\rightarrow \infty,\qquad c_4^k-c_3^k\rightarrow \infty,\qquad 0<c_3^k<M_1.
\label{conditions.good2}
  \end{align}

Owing to Lemma~\ref{lem:compactness} one can extract a subsequence from $\{m_{good}^k\}$ (not relabeled) with a limit $m^0\in A(\Omega).$
Let us now prove that conditions (\ref{conditions.good1}) and (\ref{conditions.good2}) imply that $m^0\in \tilde A(\Omega).$ We have for any fixed $R>0,$

\begin{align*}
\int_{-R}^{R}|\bar m_1^0-\bar  m_{good1}^k|\ud x&=\frac{1}{|\omega|}\int_{-R}^{R}\bigg|\int_{\omega}( m_1^0-m_{good1}^k)\ud y\ud z\bigg|\ud x\\
&\leq\frac{1}{|\omega|}\int_{-R}^{R}\int_{ \omega}| m_1^0-m_{good1}^k|\ud y\ud z\ud x\\
&\leq\frac{1}{|\omega|}\Bigg(2R|\omega|\cdot\int_{[-R,R]\times \omega}|m_1^0-m_{good1}^k|^2\ud\xi\Bigg)^{\frac{1}{2}}\\
&=\sqrt{\frac{2R}{|\omega|}}\cdot \|m_1^0-m_{good1}^k\|_{L^2([-R,R]\times \omega)}\rightarrow 0
\end{align*}
as $k\to\infty$ because of the strong convergence $m_{good}^k\rightarrow  m^0$  in $L_{loc}^2(\Omega)$. Therefore a subsequence of $\{\bar m_{good1}^k(x)\}$ converges pointwise to $\bar m_1^0(x)$ almost everywhere in $[-R,R].$ Giving $R$ all natural values and applying a diagonal argument we establish that a subsequence of $\{\bar m_{good1}^k(x)\}$ converges pointwise to $\bar m_1^0(x)$ almost everywhere in $\mathbb{R},$ therefore
\begin{equation}
\label{cond.m0}
\bar m_1^0(x)\leq -\frac{1}{2}\quad\text{ a.e. in}\quad (-\infty,0)\quad\text{ and}\quad \bar m_1^0(x)\geq\frac{1}{2} \quad\text{ a.e. in}\quad  [M_1,+\infty)
\end{equation}
Let us now show that conditions $E(m^0)<\infty$ and (\ref{cond.m0}) imply $m^0\in \tilde A(\Omega).$ We have by the triangle inequality
$$\|\nabla (m^0-\bar e)\|_{L^2(\Omega)}^2\leq 2\|\nabla m^0\|_{L^2(\Omega)}^2+2\|\nabla \bar e\|_{L^2(\Omega)}^2\leq 2E(m^0)+4|\omega|<\infty, $$
thus it remains to prove that $m^0-\bar e\in L^2(\Omega)$. We have again by the triangle inequality and by Lemma~\ref{lem:ineq.m.bar.m},
\begin{align*}
 \|m^0-\bar e\|_{L^2(\Omega)}^2&\leq 2\|\bar m^0-\bar e\|_{L^2(\Omega)}^2+2\|m^0-\bar m^0\|_{L^2(\Omega)}^2\\
 &\leq  2C_pd^2\|\nabla m^0\|_{L^2(\Omega)}^2+2\|\bar m^0-\bar e\|_{L^2(\Omega)}^2\\
 &\leq 2C_pd^2E(m^0)+2\|\bar m^0-\bar e\|_{L^2(\Omega)}^2,
 \end{align*}
 thus it remains to prove that $\bar m^0-\bar e\in L^2(\Omega)$. One can assume without loss of generality that $M_1\geq1$ in (\ref{cond.m0}). We calculate,
\begin{align*}
\int_\Omega |\bar m^0-\bar e|^2=\int_{[-1,M_1]\times\omega}|\bar m^0-\bar e|^2+\int_{[-\infty,-1]\times\omega}|\bar m^0-\bar e|^2+\int_{[M_1,\infty]\times\omega}|\bar m^0-\bar e|^2=I_1+I_2+I_3.
\end{align*}
The estimation of $I_1,$ $I_2$ and $I_3$ is straightforward:

$$I_1\leq 4(1+M_1)|\omega|.$$
Due to condition (\ref{cond.m0}) and Lemma~\ref{lem:ineq.m.bar.m} we have,
\begin{align*}
I_2&=\int_{[-\infty,-1]\times\omega}(1+|\bar m^0|^2+2\bar m_1^0)\\
&=2\int_{[-\infty,-1]\times\omega}(1+\bar m_1^0)+\int_{[-\infty,-1]\times\omega}(|m^0|^2-|\bar m^0|^2)\\
&\leq 2\int_{[-\infty,-1]\times\omega}(1+\bar m_1^0)(1-\bar m_1^0)+C_pd^2\int_{[-\infty,-1]\times\omega}|\nabla m^0|^2\\
&=2\int_{[-\infty,-1]\times\omega}(|m^0|^2-|\bar m^0|^2)+2\int_{[-\infty,-1]\times\omega}(|\bar m_2^0|^2+|\bar m_3^0|^2)+C_pd^2\int_{[-\infty,-1]\times\omega}|\nabla m^0|^2\\
&\leq 3C_pd^2\int_{[-\infty,-1]\times\omega}|\nabla m^0|^2+2\int_{[-\infty,-1]\times\omega}(|\bar m_2^0|^2+|\bar m_3^0|^2).
\end{align*}
Analogues analysis for $I_3$ gives
 $$I_3\leq3C_pd^2\int_{[M_1,\infty]\times\omega}|\nabla m^0|^2+2\int_{[M_1,\infty]\times\omega}(|\bar m_2^0|^2+|\bar m_3^0|^2).$$
Therefore combining the estimates for $I_1$ $I_2$ and $I_3$ and taking into account Lemma~\ref{lem:m2.m.3.bdd.E} we discover $I_1+I_2+I_3<\infty$
as wished. CASE1 is now established.\\

\textbf{CASE2.} \textit{There are some second type intervals in $B_n$ for some $n.$}\\

Removing all the second type yellow intervals from the real line we can regard the rest as a real line without gaps simply by shifting all the intervals to the left hand side such that after that operation no overlap occurs and there is no gap left. Precisely, we shift each interval to the left by a factor equal to the sum of the lengths of the gaps between that interval and $-\infty.$ During that operation we unify the black and red intervals with the neighboring intervals of the same color but we regard the possible neighboring first type yellow intervals as separate. We get a situation like in CASE1 and therefore we can prove the existence of a "good" subsequence. It is easy to show that since that sum of the lengths of the second type yellow intervals in each family is bounded by the same constant then the in Lemma~\ref{lem:compactness} described limit of the obtained "good" subsequence will belong to $\tilde A(\Omega)$ and
hence will be an energy minimizer in $\tilde A(\Omega)$. The proof is complete now.

\section{The stability of minimizers}

Throughout this section we will consider a sequence of domain-magnetization-energy triples $(\Omega_n, m^n,E(m^n))_{n\in\mathbb{N}}$ such that $\Omega_n=\mathbb R\times(d_n\cdot\omega),$ $m^n\in\tilde A(\Omega_n),$  $d_n\to 0$ and $\lim_{n\to\infty}\frac{E(m^n)}{d_n^2}=\min_{m\in A_0}E_0(m),$
i.e., $\{m^n\}$ is a sequence of almost minimizers. Assume furthermore that $\omega$ has 180 degree rotational symmetry and the matrix $M_\omega$ has three different eigenvalues, i.e., $\alpha_2\neq\alpha_3,$ hence one can assume without loss of generality, that $\alpha_2<\alpha_3.$ Note that due to (\ref{almost.min}) we have
\begin{equation}
\label{E(m^n)is.bdd}
E(m^n)\leq Cd_n^2\quad\text{for all}\quad n.
\end{equation}

\begin{Lemma}

\label{morms of avrages converges to the norm of limit}
If $\{\acute m^n\}$ converges to some $m^0(x)\in\tilde A(\Omega)$ in the sense of Definition~\ref{notion of convergence}, then

\begin{itemize}

\item[(i)] $ \lim_{n\to\infty}\|\nabla \acute{\bar m}^n\|_{L^2(\Omega)}=\|\nabla m^0\|_{L^2(\Omega)},$
\item[(ii)] $ \lim_{n\to\infty}\|\acute{\bar m}_2^n\|_{L^2(\Omega)}=\|m_2^0\|_{L^2(\Omega)},\qquad  \lim_{n\to\infty}\|\acute{\bar m}_3^n\|_{L^2(\Omega)}=\|m_3^0\|_{L^2(\Omega)}.$

\end{itemize}
\end{Lemma}

\begin{proof}

The inequality $ \liminf_{n\to\infty}\|\nabla \acute{\bar m}^n\|_{L^2(\Omega)}\geq\|\nabla m^0\|_{L^2(\Omega)}$ is trivial, while the inequality
$ \liminf_{n\to\infty}\|\acute{m}_2^n\|_{L^2(\Omega)}\geq\|m_2^0\|_{L^2(\Omega)}$ follows from the convergence $m_2^n\to m_2^0$ in $L_{loc}^2(\Omega).$
We have furthermore by Lemma~\ref{lem:ineq.m.bar.m} and by (\ref{E(m^n)is.bdd}) that,
\begin{align*}
\|\acute m_2^n-\acute{ \bar m}_2^n\|_{L^2(\Omega)}^2&=\frac{1}{d_n^2}\|m_2^n-\bar m_2^n\|_{L^2(\Omega_n)}\\
&\leq C_p\|\nabla m^n\|_{L^2(\Omega_n)}^2\\
&\leq  C_pCd_n^2,
\end{align*}
thus
\begin{equation}
\label{acute.acute.bar}
\|\acute m_2^n-\acute{ \bar m}_2^n\|_{L^2(\Omega)}\to 0.
\end{equation}
Therefore we get $ \liminf_{n\to\infty}\|\acute{\bar m}_2^n\|_{L^2(\Omega)}\geq\|m_2^0\|_{L^2(\Omega)}$ and a similar inequality for $m_3$
is also fulfilled. It remains to only show the opposite inequalities with $\limsup.$ It is clear that
$\|\nabla \acute{\bar m}^n\|_{L^2(\Omega)}\leq \|\nabla \acute{m}^n\|_{L^2(\Omega)},$ thus it suffices to prove that
$ \limsup_{n\to\infty}\|\nabla \acute{ m}^n\|_{L^2(\Omega)}\leq\|\nabla m^0\|_{L^2(\Omega)}.$
Assume now in contradiction that one of the three inequalities with $\limsup$, we intend to prove, fails. Therefore we have owing to (\ref{lower.bound.E}), that for some $\delta>0$ there holds,
\begin{align*}
\limsup_{n\to\infty}\frac{E(m^n)}{d_n^2}&\geq\max\bigg(\limsup_{n\to\infty}\|\nabla \acute m^n\|_{L^2(\Omega)}^2+\liminf_{n\to\infty}
\alpha_2\|\bar m_2^n\|_{L^2(\mathbb R)}^2+\liminf_{n\to\infty}\alpha_3\|\bar m_3^n\|_{L^2(\mathbb R)}^2,\\
 &\liminf_{n\to\infty}\|\nabla \acute m^n\|_{L^2(\Omega)}^2+\limsup_{n\to\infty}
\alpha_2\|\bar m_2^n\|_{L^2(\mathbb R)}^2+\liminf_{n\to\infty}\alpha_3\|\bar m_3^n\|_{L^2(\mathbb R)}^2,\\
 &\liminf_{n\to\infty}\|\nabla \acute m^n\|_{L^2(\Omega)}^2+\liminf_{n\to\infty}
\alpha_2\|\bar m_2^n\|_{L^2(\mathbb R)}^2+\limsup_{n\to\infty}\alpha_3\|\bar m_3^n\|_{L^2(\mathbb R)}^2\bigg)\\
 &\geq E_0(m^0)+\delta\\
 &\geq \min_{m\in A_0}E_0(m)+\delta,
\end{align*}
which contradicts (\ref{almost.min}). The lemma is proved now.

\end{proof}

\begin{Corollary}
\label{norms convergs to the norm of the limit}
Let $\{m^n\}$ and $m^0$ be as in Lemma \ref{morms of avrages converges to the norm of limit}. Then
\begin{itemize}
\item[(i)] $ \lim_{n\to\infty}\|\acute{m}_2^n\|_{L^2(\Omega)}=\|m_2^0\|_{L^2(\Omega)},\qquad\lim_{n\to\infty}\|\acute{m}_3^n\|_{L^2(\Omega)}=\|m_3^0\|_{L^2(\Omega)}.$
\end{itemize}
\end{Corollary}

\begin{proof}
It follows from Lemma~\ref{morms of avrages converges to the norm of limit} and equality (\ref{acute.acute.bar}).
\end{proof}

\begin{Lemma}
\label{strong convergence1}

Let $\{m^n\}$ and $m^0$ be as in Lemma~\ref{morms of avrages converges to the norm of limit}. Then
\begin{itemize}

\item[(i)]$\lim_{n\to\infty}\|\nabla \acute m^n-\nabla m^0\|_{L^2(\Omega)}=0$

\item[(ii)] $\lim_{n\to\infty}\|\acute m_2^n- m_2^0\|_{L^2(\Omega)}=0, \qquad
\lim_{n\to\infty}\|\acute m_3^n- m_3^0\|_{L^2(\Omega)}=0.$

\end{itemize}

\end{Lemma}

\begin{proof}

The inequality $\liminf_{n\to\infty}\|\nabla \acute m^n\|_{L^2(\Omega)}\geq \|\nabla  m^0\|_{L^2(\Omega)}$ ia a consequence of the weak convergence $
\nabla \acute m^n\rightharpoonup\nabla m^0.$ The opposite inequality $\limsup_{n\to\infty}\|\nabla \acute m^n\|_{L^2(\Omega)}\leq \|\nabla  m^0\|_{L^2(\Omega)}$ has been proven in the proof of Lemma~\ref{morms of avrages converges to the norm of limit}. Therefore
$\limsup_{n\to\infty}\|\nabla \acute m^n\|_{L^2(\Omega)}=\|\nabla  m^0\|_{L^2(\Omega)}$ which combined with the weak convergence
$\acute m^n \rightharpoonup m^0$ gives $(i)$.
Fix now $l>0.$ We have by virtue of Corollary \ref{norms convergs to the norm of the limit},

\begin{align*}
\limsup_{n\to\infty}\int_{\Omega}|\acute m_2^n- m_2^0|^2&\leq
\limsup_{n\to\infty}\int_{[-l,l]\times\omega}|\acute m_2^n- m_2^0|^2+\limsup_{n\to\infty}\int_{\Omega\setminus([-l,l]\times\omega)}|\acute m_2^n- m_2^0|^2\\
&\leq 2\limsup_{n\to\infty}\int_{\Omega\setminus([-l,l]\times\omega)}\big(|\acute m_2^n|^2+| m_2^0|^2\big)\\
&\leq 2\limsup_{n\to\infty}\int_{\Omega}\big(|\acute m_2^n|^2+|m_2^0|^2\big)-2\liminf_{n\to\infty}\int_{[-l,l]\times\omega}\big(|\acute m_2^n|^2+| m_2^0|^2\big)\\
&=4|\omega|\int_{\mathbb{R}\setminus[-l,l]}| m_2^0(x)|^2\ud x.
\end{align*}
From the arbitrariness of $l$ we get the validity of the first equality in $(ii).$ The proof of the second equality in $(ii)$ is straightforward.

\end{proof}

\begin{Lemma}
\label{strong convergence2}
Let $\{m^n\}$ and $m^0$ be as in Lemma~\ref{morms of avrages converges to the norm of limit}. Assume in addition that for some $N\in\mathbb{N}$ and $l>0$ we have for all $n\geq N$
$$\bar m_1^n(x)\leq 0, \ x\in(-\infty,-l] \ \ \text{and}\ \ \ \bar m_1^n(x)\geq 0, \ x\in[l,+\infty).$$
Then
$$\lim_{n\to\infty}\|\acute m^n-m^0\|_{H^1(\Omega)}=0.$$
\end{Lemma}

\begin{proof}

By Lemma~\ref{strong convergence1} it suffices to show that $\lim_{n\to\infty}\|\acute m_1^n-m_1^0\|_{L^2(\Omega)}=0.$
Since $m^0(x)\in \tilde A(\Omega)$ then due to Remark~\ref{rem:lim.bar.m1=1} there exists $l_1>0$ such that
$$m_1^0(x)\leq-\frac{1}{2},\qquad x\in(-\infty, l_1]\qquad \text{and}\qquad m_1^0(x)\geq\frac{1}{2},\qquad x\in[l_1, +\infty).$$
For any fixed $l_2>\max(l,l_1)$ we have,

$$\int_{\Omega}|\acute m_1^n-m_1^0|^2=\int_{[-l_2,l_2]\times\omega}|\acute m_1^n-m_1^0|^2+\int_{\Omega\setminus([-l_2,l_2]\times\omega)}|\acute m_1^n-m_1^0|^2.$$
The first summand converges to zero and we have furthermore that
$\|\acute m_1^n-\acute{\bar m}_1^n\|_{L^2(\Omega)}\to 0,$ thus it suffices to show that
$$\lim_{n\to\infty}\int_{\Omega\setminus([-l_2,l_2]\times\omega)}|\acute{\bar m}_1^n-m_1^0|^2=0.$$
For $n\geq N$ we have
\begin{align*}
\int_{\Omega\setminus([-l_2,l_2]\times\omega)}|\acute{\bar m}_1^n-m_1^0|^2&\leq
\int_{\Omega\setminus([-l_2,l_2]\times\omega)}\big||\acute{\bar m}_1^n|^2-|m_1^0|^2\big|\\
&\leq\int_{\Omega\setminus([-l_2,l_2]\times\omega)}\big||\acute{\bar m}_1^n|^2-|m_1^n|^2\big|+
 \int_{\Omega\setminus([-l_2,l_2]\times\omega)}\big||\acute m_1^n|^2-|m_1^0|^2|.
\end{align*}
 The first summand converges to zero, for the second summand we have by Lemma~\ref{norms convergs to the norm of the limit}

\begin{align*}
\limsup_{n\to\infty}\int_{\Omega\setminus([-l_2,l_2]\times\omega)}\big||\acute m_1^n|^2-|m_1^0|^2\big|
&\leq\limsup_{n\to\infty}\int_{\Omega\setminus([-l_2,l_2]\times\omega)}(|\acute m_2^n|^2+|\acute m_3^n|^2+|m_2^0|^2+|m_3^0|^2)\\
&\leq 2\int_{\Omega\setminus([-l_2,l_2]\times\omega)}(|m_2^0|^2+|m_3^0|^2),
\end{align*}
 which converges to zero as $l_2$ goes to infinity.

\end{proof}

\begin{Lemma}
\label{bar m and [b^1,b^2] lemma}
Let $0<\epsilon<1$ and let the sequence of intervals $\big([b_n^1, b_n^2]\big)_{n\in\mathbb N}$ be such that
$$\bar m_1^n(b_n^1)=-1+\epsilon,\quad\bar m_1^n(b_n^2)=1-\epsilon.$$
Then for sufficiently big $n$ there holds
\begin{align*}
\bar m_1^n(x)&<-1+2\epsilon,\quad x\in(-\infty, b_n^1],\quad\bar m_1^n(x)>1-2\epsilon,\quad x\in[b_n^2,+\infty),\\
-1+\frac{\epsilon}{2}&<m_1^n(x)<1-\frac{\epsilon}{2},\quad x\in [b_n^1,b_n^2].
\end{align*}

\end{Lemma}

\begin{proof}

Assume in contradiction that for a subsequence $\{n_k\}$(not relabeled) there is a point $b_n^3\in(-\infty, b_n^1)$ such that
$\bar m_x^n(b_n^3)\geq -1+2\epsilon.$ Since $\bar m_1^n(-\infty)=-1$ and $\bar m_1^n$ is continuous we can without loss of generality assume that
 $\bar m_1^n(b_n^3)=-1+2\epsilon.$  Utilizing Lemma~\ref{lemma with f} for the intervals $(-\infty, b_n^3],$ $[b_n^3,b_n^1],$ $[b_n^1, +\infty)$ and
 (\ref{lower.bound.E}) we discover,
 \begin{align*}
 \frac{E(m^n)}{d_n^2}&\geq\int_{\Omega}|\nabla \acute m^n|^2+\alpha_2\int_{\mathbb R}|m_2^n|^2+\alpha_3\int_{\mathbb R}|m_2^n|^2+o(1)\\
 &\geq 2\sqrt{\alpha_2|\omega|}\bigg(|2\epsilon|+|\epsilon|+|2-2\epsilon|\bigg)+o(1)\\
 &=(4+2\epsilon)\min_{m\in A_0}E_0(m)+o(1),
 \end{align*}
  which contradicts the almost minimizing property of $\{m^n\}.$ Similarly we get the bounds near $\infty$ and in $[b_n^1,b_n^2].$

\end{proof}

\subsection{Proof of Theorem~\ref{th:almost.minimizers}}

\begin{proof}

The proof splits into some steps:\\

\textbf{Step1.} Let us prove that if a sequence of magnetizations converges to some $m^0\in\tilde A(\Omega)$ in the sense of Definition \ref{notion of convergence} and satisfies conditions (\ref{almost.min}) and $\bar m_2^n(x_0)\geq 0$ for some $x_0\in\mathbb R$ and for big $n$ then $m_2^0(x_0)\geq 0.$

We have due to (\ref{almost.min}), that
$$\int_{\Omega_n}|\partial_x \bar m^n|^2\leq\int_{\Omega_n}|\partial_x m^n|^2\leq Cd_n^2,$$ thus
$$\int_{\mathbb R}|\partial_x \bar m^n(x)|^2\ud x\leq \frac{C}{|\omega|}$$
which yields that the sequence $\{\bar m^n\}$ is equicontinuous in $\mathbb R,$ and therefore by the Arzela-Ascoli theorem $\{\bar m^n(x)\}$ has a subsequence with a uniform limit in the interval $[x_0-1,x_0+1].$ It is trivial that the limit is $m^0,$ and thus $\bar m_2^n(x_0)\geq 0$ yields $m_2^0(x_0)\geq 0.$
Evidently, the same sing preserving property holds for the first and the third components of $\bar m^n$ and also for the opposite sign. This means in particular that if $\bar m_1^n(x_0)=0$ for big $n$ then $m_1^0(x_0)=0.$\\

\textbf{Step2.} In the second step we construct the sequences $\{T_n\}$ and $\{R_n\}.$ Note first, that the change of variables mentioned in the theorem translates the domain $\Omega$ to itself and  preserves the energy, thus the minimization problem (\ref{minimization problem}) is invariant under that kind of transformations. Let us now evaluate the constant in estimate (\ref{(a,b).finite.estimate}).
The constant $C_1$ in (\ref{(a,b).finite.estimate}) comes from Lemma~\ref{lem:m2.m.3.bdd.E} and is given by
$$C_1=\frac{2d_n^2E(m^n)}{\alpha_2d_n^2}\leq \frac{2Cd_n^2}{\alpha_2},$$
for big $n.$ Thus we get
\begin{align*}
M(\alpha,\beta,\rho,\omega_n,E(m^n))&=\frac{1}{|\omega_n|}\left(\frac{E(m^n)}{(\alpha-\beta)^2}+\frac{C_1+C_pd_n^2E(m^n)}{1-\rho^2}\right)\\
&\leq \frac{1}{|\omega|}\left(\frac{C}{(\alpha-\beta)^2}+\frac{\frac{2C}{\alpha_2}+C_pCd_n^2}{1-\rho^2}\right)\\
&\leq M_1
\end{align*}

uniformly in $n.$ Next we choose the intervals $[b_n^1,b_n^2]$ to be as in Lemma~\ref{bar m and [b^1,b^2] lemma} with $\epsilon=\frac{1}{3},$ which is possible due to the continuity of $\bar m_1^n$ and the fact that $\bar m_1^n(\pm\infty)=\pm1.$ Owing to Lemma~\ref{bar m and [b^1,b^2] lemma} we get
\begin{align}
\label{b1b2.est}
\bar m_1^n(x)&<-\frac{1}{3},\quad x\in(-\infty, b_n^1],\quad\bar m_1^n(x)>\frac{1}{3},\quad x\in[b_n^2,+\infty),\\
-\frac{5}{6}&<m_1^n(x)<\frac{5}{6},\quad x\in [b_n^1,b_n^2].
\end{align}
Therefore, we obtain by the uniform estimate on $M(\alpha,\beta,\rho,\omega_n,E(m^n))$ and by the estimate (\ref{(a,b).finite.estimate}) of Lemma~\ref{lem:oscilation.preventing} that for sufficiently big $n$ there holds,
\begin{equation}
\label{b2-b1.bdd}
b_n^2-b_n^1\leq M_1.
\end{equation}

 Let now $x_n\in [b_n^1,b_n^2]$ be such that $\bar m_1^n(x_n)=0$. For any $n\in\mathbb N$ we choose $T_n$ to be the translation by $x_n$ and the rotation $R_n$ to be the identity if $\bar m_2^n(x_n)\geq 0$ and the rotation by $180$ degree otherwise. In the last step we prove that the whole sequence
 $\{\acute{\tilde m}^n\}$ converges to $m^\omega$ in $H^1(\Omega).$\\

 \textbf{Step3.} For convenience of notation we will omit the "tilde" in $\acute{\tilde m}^n.$ We are now ready to prove that $\|\acute m^n-m^\omega\|_{H^1(\Omega)}\to 0$ as $n\to\infty.$ Assume in contradiction that for a subsequence (not relabeled) $\|\acute m^n-m^\omega\|_{H^1(\Omega)}\geq \delta>0$ for some $\delta.$ Like in the proof of Lemma~\ref{lem:compactness} we can show that a subsequence of $\{\acute m^n\}$ converges to some $m^0$ in the sense of Definition~\ref{notion of convergence}. By the $\Gamma$-convergence theorem we then have $E_0(m^0)\leq \liminf_{n\to\infty}\frac{E(m^n)}{d_n^2}$ thus
\begin{equation}
\label{m0.mins.E0}
E_0(m^0)=\min_{m\in A_0}E_0(m).
\end{equation}

 Next we have by the sign-preserving property of Step1 and by bounds (\ref{b1b2.est})--(\ref{b2-b1.bdd}), that

\begin{equation}
\label{m0.M1.est}
\bar m_1^0(x)\leq-\frac{1}{3},\quad x\in(-\infty, -M_1],\quad\bar m_1^0(x)\geq\frac{1}{3},\quad x\in[M_1,+\infty).
\end{equation}
Invoking now Remark~\ref{rem:lim.m0} and the properties (\ref{m0.mins.E0}) and (\ref{m0.M1.est}) we discover $m_1^0(\pm\infty)=\pm1,$ which yields
\begin{equation}
\label{m0.in.A0}
m^0\in A_0,
\end{equation}
 i.e., $m_0$ is a minimizer of the minimization problem (\ref{min.prob.E0}). Again, by the sign-preserving property we have $m_1^0(0)=0$ and $m_2(0)\geq 0,$ thus by
 the analysis on the minimization problem (\ref{min.prob.E0}) in Appendix, we establish that actually $m^0$ and $m^\omega$ coincide. Note, finally, that the requirements of
Lemma~\ref{strong convergence2} are satisfied, thus we get
 $$\lim_{n\to\infty}\|\acute m^n-m^\omega\|_{H^1(\Omega)}=\lim_{n\to\infty}\|\acute m^n-m^0\|_{H^1(\Omega)}=0,$$ which is s contradiction. The theorem is proven now.

\end{proof}

We mention that it is easy to see that any rectangle that is not a square and any ellipse that is not a circle satisfies the condition $0<\alpha_2<\alpha_3.$
This condition shows that the cross section $\omega$ does not have many rotational symmetries in some sense. For instance, if $\omega$ has a 90 degree rotational symmetry, then one can show that $\alpha_2=\alpha_3.$ It is also worth mentioning that one can prove a modified version of Theorem~\ref{th:almost.minimizers} in the case when $\omega$ is a disc or a canonical polygon with even number of vertices, namely due to the symmetry it is not true that any of the rotations $R_n$ is either the identity the rotation by $180$ degree, but one can prove their existence. In conclusion we state that Theorem~\ref{th:almost.minimizers} shows that in thin wires energy minimizers with a 180 degree domain wall are transverse (Ne\'el) walls that have the shape of $m^\omega.$

\begin{Remark}
The stability result holds for circular cross sections and cross sections that are canonical polygons.
\end{Remark}
\begin{proof}
It is straightforward to check that the construction of the rotation $R_n$ and translation $T_n$ can be done exactly in the same way as in the proof of Theorem~\ref{th:almost.minimizers}, although the cross sections do not have the required asymmetry.
\end{proof}

\appendix
\section{Appendix}

In appendix we recall some well-know facts, in particular the study of the minimization problem $\min_{m\in A_0}E_0(m).$
We start with a simple lemma.
\begin{Lemma}
\label{lem:mag.m_1.m_2}
For any magnetizations $m_1,m_2\in A(\Omega)$ there holds
$$|E_{mag}(m_1)-E_{mag}(m_2)|\leq \|m_1-m_2\|_{L^2(\Omega)}^2+2\|m_1-m_2\|_{L^2(\Omega)}
\sqrt{E_{mag}(m_1)}$$ \\
\end{Lemma}

\begin{proof}
The proof is elementary and can be found in [\ref{bib:Kuehn}].
\end{proof}

\begin{Lemma}
\label{lem:bound.Green.function}

For any $0<s\leq r$ denote $R(s,r)=[-s,s]\times [-r,r].$ Then for all points $(y_1,z_1)\in \mathbb R^2$ there holds
$$I=\int_{R(s,r)}\frac{\ud y\ud z}{\sqrt{(y-y_1)^2+(z-z_1)^2}}< 10s\Big(1+\ln\frac{r}{s}\Big).$$

\end{Lemma}

\begin{proof}

It is clear that if we replace the point $(y_1,z_1)$ by its closest point to $R(s,r),$ then the integral may only increase. Thus one can without loss of generality assume that $(y_1,z_1)\in R(s,r).$ We have that

\begin{align*}
I\leq \int_{R(2s,2r)}\frac{\ud y\ud z}{\sqrt{y^2+z^2}}&=\int_{R(2s,2s)}\frac{\ud y\ud z}{\sqrt{y^2+z^2}}+
\int_{R(2s,2r)\setminus R(2s,2s)}\frac{\ud y\ud z}{\sqrt{y^2+z^2}}\\
&\leq\frac{1}{4}\int_{D_{4\sqrt2 s}(0)}\frac{\ud y\ud z}{\sqrt{y^2+z^2}}+8s\int_{2s}^{2r}\frac{\ud y}{y}\\
&=2\sqrt2 \pi s+8s\ln\frac{r}{s}\\
&<10s\Big(1+\ln\frac{r}{s}\Big).
\end{align*}
\end{proof}

\begin{Lemma}
\label{lemma with f}
 Assume that $\omega\subset\mathbb{R}^2$ is a bounded Lipschitz domain. Then for any interval $(a,b)\subset\mathbb{R},$ positive $\alpha$ and a unit vector field  $f\in H^1\big((a,b)\times\omega, \mathbb{R}^3\big)$ there holds:
$$\int_{(a,b)\times\omega}|\partial_x f|^2+\alpha^2\int_{(a,b)\times\omega}(|f_2|^2+|f_3|^2)\geq
2\alpha|\omega||\bar f_1(a)-\bar f_1(b)|.$$
(The endpoints $a$ and $b$ can take values $-\infty$ and $\infty$ respectively).
\end{Lemma}

\begin{proof}

Fix a point $(y,z)\in\omega$ and consider the vector field $f$ on the segment with endpoints $(a,y,z)$ and $(b,y,z).$ Being an $H^1$ vector field, $f$ must be absolutely continuous on that segment as a function of one variable, thus denoting
$$f_1(x,y,z)=\sin\varphi(x), f_2(x,y,z)=\cos\varphi(x)\cos\theta(x), f_3(x,y,z)=\cos\varphi(x)\sin\theta(x)$$ we obtain that $\varphi$ and $\theta$ are differentiable a.e. in $[a,b].$
Therefore we can calculate,
\begin{align*}
\int_{(a,b)\times(y,z)}|\partial_x f(\xi)|^2\ud x&+\alpha^2\int_{(a,b)\times(y,z)}(|f_2(\xi)|^2+|f_3(\xi)|^2)\ud x\\
&=\int_a^b(\varphi'^2(x)+\theta'^2(x)\cos^2\varphi(x))\ud x+\alpha^2\int_a^b\cos^2\varphi(x)\ud x\\
&\geq\int_a^b(\varphi'^2(x)\ud x+\alpha^2\int_a^b\cos^2\varphi(x)\ud x\\
&\geq 2\alpha\bigg|\int_a^b\varphi\prime(x)\cos\varphi(x)\ud x\bigg|\\
&=2\alpha|f_1(a,y,z)-f_1(b,y,z)|.
\end{align*}
 An integration of the obtained inequality over $\omega$ completes the proof.
\end{proof}

Recall now that analogues to the proof of the above lemma one can determine the minima of the energy functional
$$E_{\alpha}(m)=\int_{\mathbb{R}}|\partial_x m(x)|^2\ud x+\alpha\int_{\mathbb{R}}(| m_y(x)|^2+|m_z(x)|^2)\ud x$$
for any $\alpha>0$ in the admissible set
 $$A_0=\{m\colon \mathbb{R}\to\mathbb{R}^3 \ : \ |m|=1, m(\pm\infty)=\pm1 \}.$$
 The minimizer is unique up to a translation in the $x$ coordinate and a rotation in the $OYZ$ plane and is given by
 \begin{equation}
 \label{limit problem minimizer}
 m^{\alpha,\beta}=\bigg(\frac{e^{2\sqrt{\alpha}x}\cdot\beta-1}{e^{2\sqrt{\alpha}x}\cdot\beta+1},\  \frac{2\sqrt\beta e^{\sqrt{\alpha}x}}{e^{2\sqrt{\alpha}x}\cdot\beta+1}\cos\theta,\ \frac{2\sqrt\beta e^{\sqrt{\alpha}x}}{e^{2\sqrt{\alpha}x}\cdot\beta+1}\sin\theta\bigg).
 \end{equation}
 Set $m^{\alpha}:=m^{\alpha,1}$ to have $m_x^{\alpha}(0)=0.$ The minimal value of $E_\alpha$ in $A_0$ will be $4\sqrt{\alpha}.$ Let us now find the minimal value and the minima of the reduced problem for any $\omega$ under the condition
 \begin{equation}
 \label{alpha2.alpha3}
 0<\alpha_2<\alpha_3,
 \end{equation}
 i.e., consider the minimization problem
 \begin{equation}
 \label{min.prob.E0}
 \min_{m\in A_0}E_0(m).
 \end{equation}
  Observe that,
 \begin{align*}
E_0(m)&=|\omega|\int_{\mathbb R}|\partial_x m(x)|^2\ud x+\alpha_2\int_{\mathbb R}m_2^2(x)+\alpha_3\int_{\mathbb R}m_3^2(x)\\
&\geq |\omega|\int_{\mathbb R}|\partial_x m(x)|^2\ud x+\alpha_2\int_{\mathbb R}m_2^2(x)+\alpha_2\int_{\mathbb R}m_3^2(x)\\
&\geq 4\sqrt{|\alpha_2\omega|,}
\end{align*}
and the minimum is realized by
\begin{equation}
\label{fixed minimizer of the limit enargy}
m^\omega=\bigg(\frac{e^{2\sqrt{\alpha_\omega}x}-1}{e^{2\sqrt{\alpha_\omega}x}+1},\  \frac{2 e^{\sqrt{\alpha_\omega}x}}{e^{2\sqrt{\alpha_\omega}x}+1},\ 0\bigg),
\end{equation}
where $\alpha_\omega=\frac{\alpha_2}{|\omega|}.$ All other minimizers of $E_0$ are obtained via translations and 180 degree rotations of $m^\omega.$ \\

 \textbf{\large{Acknowledgement}}\\

 The author is very grateful to his Ph.D. thesis supervisor Dr. Prof. S. M\"uller for suggesting the topic and for many fruitful discussions.

\end{document}